\documentclass[12pt]{amsart}
\usepackage{amsmath,amsthm,amsfonts,tikz}
\usepackage{hyperref}
\usepackage{enumerate}
\usepackage[shortlabels]{enumitem}
\setlist[enumerate,1]{label={\upshape(\roman*)}}

\DeclareMathOperator{\Spec}{Spec}

\newcommand{\h}{\mathfrak{H}}
\newcommand{\Ho}{\mathfrak{h}}
\newcommand{\C}{\mathbb{C}}
\newcommand{\Z}{\mathbb{Z}}
\newcommand{\N}{\mathbb{N}}
\newcommand{\R}{\mathbb{R}}
\newcommand{\allone}{\mathbf{1}}

\newcommand{\ep}{\varepsilon}
\newcommand{\cD}{\mathcal{D}}
\newcommand{\cS}{\mathcal{S}}
\newcommand{\LG}{\operatorname{\mathfrak{L}}}
\newcommand{\sHg}{Hoffman signed graph}

\newcommand{\nexteqv}{\displaybreak[0]\\ &\iff}

\newcommand{\nexteq}{\displaybreak[0]\\ &=}
\newcommand{\refby}[1]{&&\text{(by (\ref{#1}))}}

\newtheorem{theorem}{Theorem}[section]
\newtheorem{lemma}[theorem]{Lemma}
\newtheorem{proposition}[theorem]{Proposition}
\newtheorem{corollary}[theorem]{Corollary}
\theoremstyle{definition}
\newtheorem{definition}[theorem]{Definition}

\title[Signed analogue of line graphs]{Signed analogue of line graphs and their smallest eigenvalues}
\author[A. L. Gavrilyuk]{Alexander L. Gavrilyuk}
\address{Pusan National University}
\email{gavrilyuk@riko.shimane-u.ac.jp}
\author[A. Munemasa]{Akihiro Munemasa}
\address{Graduate School of Information Sciences\\
Tohoku University}
\email{munemasa@math.is.tohoku.ac.jp}
\author[Y. Sano]{Yoshio Sano}
\address{Faculty of Engineering, Information and Systems\\
University of Tsukuba}
\email{sano@cs.tsukuba.ac.jp}
\author[T. Taniguchi]{Tetsuji Taniguchi}
\address{Hiroshima Institute of Technology}
\email{t.taniguchi.t3@cc.it-hiroshima.ac.jp}

\begin{document}
\keywords{}
\subjclass[2010]{05C50,05C22,15A18,15B57}
\date{\today}

\begin{abstract}
In this paper, we show that every connected signed graph with smallest
eigenvalue strictly greater than $-2$ and large enough minimum degree is
switching equivalent to a complete graph. This is a signed analogue of 
a theorem of Hoffman. The proof is based on what we call Hoffman's limit 
theorem which we formulate for Hermitian matrices, and also the
extension of the concept of Hoffman graph and line graph for the setting
of signed graphs.
\end{abstract}

\maketitle

\section{Introduction}

Let $G$ be a simple graph with the vertices-versus-edges incidence $(0,1)$-matrix $N$.
It is well known that the line graph $L(G)$ of $G$, whose vertices
are the edges of $G$ with two edges being adjacent whenever they are incident,
has adjacency matrix
\[
A(L(G))=N^{\top}N-2I_{|E(G)|},
\]
and hence its smallest eigenvalue $\lambda_{\min}$ is at least $-2$. Although this property is not exclusive,
a theorem of Cameron, Goethals, Shult, and Seidel \cite{CGSS}, which is one of the most beautiful results in algebraic graph theory,
classifies 
connected
graphs having $\lambda_{\min}\geq -2$.
Namely, such a graph $L$ on at least $37$ vertices must be a generalized line graph in the sense that
its adjacency matrix satisfies
\[
A(L)=N^{\top}N-2I,
\]
for some $(0,\pm 1)$-matrix $N$.
(A combinatorial definition of generalized line graphs can be found in \cite{MR0469826}.)
The proof relies on the classification of root systems in Euclidean space \cite{CGSS}.

Another approach, which sheds light on the structure of graphs with $\lambda_{\min}\geq -2$,
was developed by Hoffman \cite{MR0469826}. It is based on Ramsey's theorem
and a special class of vertex-colored graphs, 
which are called Hoffman graphs
by Woo and Neumaier \cite{hlg}.
Hoffman's results not only show that graphs with $\lambda_{\min}\geq -2$ and on sufficiently
large number of vertices are generalized line graphs, but also establish the existence of some limit
points of the smallest eigenvalues of a sequence of graphs $G$ with increasing minimum degree $\delta(G)$.

\begin{theorem}[{\cite{MR0469826}}]\label{thm:H}
There exists an integer valued function $h$ defined on 
the half-open interval $(-1-\sqrt{2},-1]$,
such that 
\begin{itemize}
  \item[$(1)$] for each $\lambda\in(-2,-1]$, 
  if $G$ is a connected graph with 
  $\lambda_{\min}(G)\geq\lambda$ and $\delta(G)\geq h(\lambda)$,
  then $G$ is a clique and $\lambda_{\min}(G)=-1$;
  \item[$(2)$] for each $\lambda\in(-1-\sqrt{2},-2]$,
  if $G$ is a connected graph with 
  $\lambda_{\min}(G)\geq\lambda$ and $\delta(G)\geq h(\lambda)$,
  then $G$ is a generalized line graph and $\lambda_{\min}(G)=-2$.
\end{itemize}
\end{theorem}

This means that the intervals $(-2,-1)$ and $(-1-\sqrt{2},-2)$ are ignorable
if our concern is the smallest eigenvalues of graphs with sufficiently large minimum degree.
Note that Hoffman~\cite{MR0469826} stated this theorem in a slightly
weaker manner, but his proof shows that the above statement is valid
(see \cite[Theorem~5.5 and Remark~5.6]{hlg}).

As a tool for the proof of Theorem \ref{thm:H}, Hoffman showed what we would call
Hoffman's limit theorem (sometimes attributed to Ostrowski and Hoffman due to
their unpublished work; see \cite{H-SIAM}). 
A very terse proof of this theorem was given
by Hoffman in \cite[Lemma~2.2]{H-GEOM}
(see also \cite[Theorem~2.14]{JKMT}).

Signed graphs can be thought of as simple graphs whose edges get labels from the set $\{+1,-1\}$.
A natural extension of $(0,1)$-matrices associated with graphs to $(0,\pm 1)$-matrices allows
to study the spectra of signed graphs. The spectral theory of signed graphs has received
much attention recently \cite{BCKW}, in particular, Problem~3.17 from \cite{BCKW} 
suggests to extend
the so-called Hoffman theory, including the above-mentioned Theorem~\ref{thm:H}, to signed graphs.

In this paper, we first provide a detailed proof of a slightly generalized
version of Hoffman's limit theorem (see Theorem \ref{lem:2.2}), which
is then used to show our main result: an analogue of Part $(1)$ of Theorem \ref{thm:H}
for signed graphs (see Section~\ref{sec:pre} for precise definitions).

\begin{theorem}\label{thm:1.1S}
There exists an integer valued function $f$ defined on 
the half-open interval $(-2,-1]$
such that, for each $\lambda\in(-2,-1]$,
if a connected signed graph $S$ satisfies $\lambda_{\min}(S)\geq\lambda$,
and $\delta(S)\geq f(\lambda)$,
then $S$ is switching equivalent to a complete graph (and hence $\lambda_{\min}(S) = - 1$).
\end{theorem}

Note that the proof of Theorem~\ref{thm:H} in \cite{MR0469826} uses Ramsey's theorem,
which produces astronomical estimates for $h(\lambda)$.
In proving Theorem~\ref{thm:1.1S} in Section \ref{s:Proof}, we introduce the notion of
signed Hoffman graphs, and then involve a structural classification of signed graphs
with smallest eigenvalue greater than $-2$ from \cite{JCTB}, which in turn
relies on the root systems. This enables one to obtain close to tight estimates for $f(\lambda)$.

In Section \ref{s:genline}, we consider a signed analogue of generalized line graphs.
It would be interesting to see whether an analogue of Part $(2)$ of Theorem \ref{thm:H} can be shown.
Note that Woo and Neumaier \cite{hlg} went on further to extend Hoffman's ideas to graphs
with smallest eigenvalue at least $\alpha$, where $\alpha\approx -2.4812$ is a zero
of the cubic polynomial $x^3+2x^2-2x-2$. Finally, Koolen, Yang and Yang \cite{KYY}
recently proved a version of Theorem~\ref{thm:H} for simple graphs
with smallest eigenvalue at least $-3$.

\section{Preliminaries}\label{sec:pre}

A \textbf{signed graph}
$S$ is a triple
$(V, E^{+}, E^{-})$
of a set $V$ of vertices,
a set $E^{+}$ of $2$-subsets
of $V$ (called $(+)$-\textbf{edges}, or \textbf{positive} edges), and
a set $E^{-}$ of $2$-subsets
of $V$ (called $(-)$-\textbf{edges}, or \textbf{negative} edges)
such that
$E^{+} \cap E^{-} = \emptyset$.
A signed graph in which $E^{-} = \emptyset$ is called an
\textbf{unsigned graph} or simply, a \text{graph}.

Let $S$ be a signed graph.
We denote
the set of vertices of $S$
by $V(S)$,
the set of $(+)$-edges of $S$
by $E^{+}(S)$,
and
the set of $(-)$-edges of $S$
by $E^{-}(S)$.
By a \textbf{subgraph}
$S' = (V(S'), E^{+}(S'), E^{-}(S'))$ of $S$ we mean a vertex
induced signed subgraph, i.e.,
$V(S') \subseteq V(S)$ and
$E^{\pm}(S')=\{\{x,y\} \in E^{\pm}(S) \mid x,y \in V(S')\}$.
If $S'$ is a subgraph of $S$, then we say that $S$
\textbf{contains} $S'$.
The \textbf{underlying graph} $U(S)$ of $S$
is the unsigned graph
$(V(S),E^+(S) \cup E^-(S))$.
The \textbf{minimum degree} $\delta(S)$ of $S$ is defined to be
the minimum degree of $U(S)$.
The signed graph $S$ is \textbf{connected} if
$U(S)$ is connected.

Two signed graphs $S$ and $S'$
are said to be \textbf{isomorphic}
if there exists a bijection $\phi: V(S) \to V(S')$
such that $\{u,v\} \in E^+(S)$ if and only if
$\{\phi(u), \phi(v) \} \in E^+(S')$ and
that $\{u,v\} \in E^-(S)$ if and only if
$\{\phi(u), \phi(v) \} \in E^-(S')$.
For a signed graph $S$, we define its
\textbf{adjacency matrix}
$A(S)$ by
\[
(A(S))_{uv}=
\begin{cases}
1  & \text{if } \{u,v\} \in E^+(S), \\
-1 & \text{if } \{u,v\} \in E^-(S), \\
0  & \text{otherwise.}
\end{cases}
\]
The eigenvalues of $S$ are defined to be those of $A(S)$.

A \textbf{switching} at a vertex $v$ is the process of swapping the signs of each edge incident to $v$.
Two signed graphs $S$ and $S'$
are said to be \textbf{switching equivalent} if there exists a subset $W \subseteq V(S)$ such that $S^\prime$ is isomorphic to the graph obtained by switching at each vertex in $W$.
Note that switching equivalent signed graphs have the same multiset of
eigenvalues.

Let $S$ be a signed graph with smallest eigenvalue at least $-2$.
A \textbf{representation} of $S$ is a mapping $\phi$ from $V(S)$ to $\R^n$
for some positive integer $n$
such that $(\phi(u),\phi(v)) = \pm 1$ if
$\{u,v\} \in E^\pm (S)$ respectively, and $(\phi(u),\phi(v))=2\delta_{u,v}$
otherwise,
where
$\delta_{u,v}$ is Kronecker's delta, i.e.,
$\delta_{u,v}=1$ if $u=v$ and $\delta_{u,v}=0$ if $u \neq v$.
Since $A(S) + 2I$ is positive semidefinite,
it is the Gram matrix of a set of vectors $\mathbf{x}_1, \dots, \mathbf{x}_m$.
These vectors satisfy $(\mathbf{x}_i, \mathbf{x}_i) = 2$ and
$(\mathbf{x}_i, \mathbf{x}_j) = 0,\pm 1$ for $i \ne j$.
Sets of vectors satisfying these conditions
determine
line systems.
We denote by $[\mathbf{x}]$ the line determined by a nonzero vector
$\mathbf{x}$, in other words, $[\mathbf{x}]$ is the one-dimensional
subspace spanned by $\mathbf{x}$.
We say that $S$ is \textbf{represented by} the line system $\cS$ if
$S$ has a representation $\phi$ such that $\cS = \{ [\phi(v)] : v \in V(S) \}$.

Below we give descriptions of three line systems, $A_n$, $D_n$ and $E_8$.
Let $\mathbf{e}_1, \dots, \mathbf{e}_n$ be an orthonormal basis for $\R^n$.
\begin{align*}
	A_n &=
	\left \{ [ \mathbf{e}_i - \mathbf{e}_j ]
	: 1 \leq i < j \leq n+1 \right \}\quad(n\geq1), \\
	D_n &= A_{n-1} \cup \left \{ [ \mathbf{e}_i + \mathbf{e}_j ]
	: 1 \leq i < j \leq n \right \}\quad(n\geq4),\\
	E_8 &= D_8 \cup
	\left \{ [ \frac{1}{2} \sum_{i=1}^8 \epsilon_i\mathbf{e}_i ]
	: \epsilon_i \in\{\pm 1\},\; \prod_{i=1}^8 \epsilon_i = 1 \right \}.
\end{align*}
These line systems are used in the following classical result of Cameron, Goethals, Shult, and Seidel.

\begin{theorem}[\cite{CGSS}]
	\label{thm:CGSS}
	Let $G$ be a connected graph with $\lambda_{\min}(G) \geq -2$.
	Then $G$ is represented by a subset of either $D_n$ or $E_8$.
\end{theorem}

Let $S$ be a signed graph represented by a line system $\cS$.
If $\cS$ can be embedded into $\Z^n$ for some $n$, then we say that $S$ is
\textbf{integrally represented} or that $S$ has an
\textbf{integral representation}.
By Theorem~\ref{thm:CGSS}, for a signed graph $S$ with
$\lambda_1(S) \geq -2$, $G$ has an integral representation if and only if
$S$ is represented by a subset of $D_n$ for some $n$.
Let $S$ be a connected signed graph with $\lambda_1(S) \geq -2$.
We call $S$ \textbf{exceptional} if it does not have an integral representation.
Clearly there are only finitely many exceptional signed graphs.

Let $S$ be a signed graph with smallest eigenvalue
greater than
$-2$. Assume that $S$ has an integral representation $\phi$ in $\mathbb{R}^n$.
This means that, with $m=|V(S)|$,
there exists an $n\times m$ matrix
\[
M=\begin{pmatrix} \mathbf{v}_1&\cdots&\mathbf{v}_m\end{pmatrix},
\]
with entries in $\Z$, such that
$(\mathbf{v}_i,\mathbf{v}_j) = \pm 1$ if $\{i,j\} \in E^\pm (S)$ respectively, and
$(\mathbf{v}_i,\mathbf{v}_j)= 2\delta_{i,j}$ otherwise.
We may assume that $M$ has no rows consisting only of zeros.
Since $\mathbf{v}_i\in\Z^n$, $\mathbf{v}_i$ has two entries equal to
$\pm1$, and all other entries $0$.
Let $H$ be the graph with vertex set $\{1,\dots,n\}$,
where vertices $i$ and $j$ are joined by the edge $k$ whenever 
$\mathbf{v}_k$ has $\pm1$ in its $i$th and $j$th positions.
Note that the graph $H$ may have multiple edges.
A graph without multiple edges is called \textbf{simple}.
We call $H$ the \textbf{representation graph} of $S$ associated with the
representation $\phi$.
Note that $H$ has no isolated vertex. If $S$ is connected, then so is $H$.

Let $S$ be an $m$-vertex connected signed graph having an integral representation
$\phi$ and smallest eigenvalue greater than $-2$.
Let $H$ be the $n$-vertex representation graph of $S$ associated with the
representation $\phi$.
Then by \cite[Lemma~5]{JCTB},
we have $n\in\{m,m+1\}$.
Moreover, if $n=m$, then $H$ is a unicyclic graph or a tree with a double edge and if $n=m+1$,
then $H$ is a tree.

For a simple graph $H$, we denote by $\LG(H)$ the line graph of $H$.
If $u$ and $v$ are adjacent vertices in a graph, then we denote the
edge $\{u,v\}$ by $uv$ for brevity.

Let $H$ be a unicyclic graph whose unique cycle $C$ has at least $4$ vertices and let
$G = \LG(H)$.
Then for each edge $e$ of $G$ there exists a unique
maximal clique that contains $e$.
For such a graph $G$, we denote by $\mathfrak C_G(e)$ the unique maximal clique of $G$ containing the edge $e$. 
Let $uu'$ be an edge of $\LG(C)$.
Define $\LG^\dag(H,uu^\prime)$ to be the signed graph $(V,E^{+},E^-)$, where $V = V(\LG(H))$,
\[
E^- = \left \{ uv \in E(\LG(H)) \mid v \in \mathfrak C_{\LG(H)}(uu^\prime) \right \}
\]
and $E^+ = E(\LG(H)) \backslash E^-$.
Observe that, for all edges $uu^\prime$ and $vv^\prime$ of $\LG(C)$, the graph $\LG^\dag(H,uu^\prime)$ is switching equivalent to $\LG^\dag(H,vv^\prime)$.

Let $H$ be a tree with a double edge $u$ and $u'$,
and let $H'=H-u'$ be the simple tree obtained from $H$ by removing $u'$.
We define $\LG(H)$ to be the
signed graph obtained from the line graph $\LG(H')$
by attaching a new vertex $u'$, and join $u'$ by $(+)$-edges to
every vertex of a clique in the neighborhood of $u$,
$(-)$-edges to every vertex of the other clique in the neighborhood of $u$.
Note that there are two different ways to assign signs to edges from $u'$,
but the resulting two signed graphs are switching equivalent.

\begin{theorem}[{\cite[Theorem~6]{JCTB}}]\label{thm:3}
Let $S$ be a connected integrally represented signed graph
having smallest eigenvalue greater than $-2$.
Let $H$ be the representation graph of $S$ for some integral representation.
Then one of the following statements holds:
\begin{enumerate}
\item\label{t1}
$H$ is a simple tree, and
$S$ is switching equivalent to the line graph $\LG(H)$,
\item\label{t2}
$H$ is unicyclic with an odd cycle, and
$S$ is switching equivalent to the line graph $\LG(H)$,
\item\label{t3}
$H$ is unicyclic with an even cycle $C$,
and $S$ is switching equivalent to
$\LG^\dag(H,uu^\prime)$
where $uu'$ is an edge of $\LG(C)$.
\item\label{t4}
$H$ is a tree with a double edge, and
$S$ is switching equivalent to $\LG(H)$.
\end{enumerate}
Conversely, if $S$ is a signed graph described by
\ref{t1}--\ref{t4}
above, then $S$ is integrally represented and has smallest
eigenvalue greater than $-2$.
\end{theorem}

\begin{corollary}\label{cor:1.4}
Let $S$ be a connected integrally represented signed graph
having smallest eigenvalue greater than $-2$.
Then there exists a tree $H$ such that $\LG(H)$ is switching equivalent to
$S$ with possibly one vertex removed.
\end{corollary}
\begin{proof}
The assertion is clear if Theorem~\ref{thm:3}\ref{t1} holds.
For the case \ref{t2} of Theorem~\ref{thm:3},
let $e$ be an edge of $H$ contained in the unique cycle of $H$.
Regarding $e$ as a vertex of $\LG(H)$, we have
$\LG(H)-e=\LG(H-e)$. Since $S$ is switching equivalent to $\LG(H)$,
$\LG(H)-e=\LG(H-e)$ is switching equivalent to $S$ with one vertex
removed. Since $H-e$ is a tree, the assertion holds.

For the case \ref{t3} of Theorem~\ref{thm:3}, we proceed in a similar manner.
Since $\LG^\dag(H,uu^\prime)-u$ has no $(-)$-edge, we have
$\LG^\dag(H,uu^\prime)-u=\LG(H-u)$, which is the
line graph of a tree.
Since $S$ is switching equivalent to $\LG^\dag(H,uu^\prime)$,
$\LG(H-u)$ is switching equivalent to $S$ with one vertex
removed.

Finally, for the case \ref{t4} of Theorem~\ref{thm:3}, let
$H$ be a tree with a double edge $u$ and $u'$.
Then $\LG(H)-u'=\LG(H-u')$ by construction.
Since $S$ is switching equivalent to $\LG(H)$,
$\LG(H-u')$ is switching equivalent to $S$ with one vertex
removed. Since $H-u'$ is a simple tree, the assertion holds.
\end{proof}

\section{Hoffman's limit theorem}
For $z\in\C$ and $\ep>0$, we define
\[\cD(z;\ep)=\{w\in\C\mid |w-z|<\ep\}.\]
By a polynomial, we mean a polynomial with coefficients in $\C$.

\begin{lemma}\label{lem:59b}
Let $(f_n(z))_{n\in\N}$ be a sequence of polynomials of bounded degree.
Suppose that this sequence converges to a nonzero polynomial $f(z)$
coefficient-wise.
Then the following statements are equivalent for $\zeta\in\C$.
\begin{enumerate}
\item\label{R1} $f(\zeta)=0$,
\item\label{R2} for every $\ep>0$, there exists $n_0(\ep)\in\N$ such that
\begin{equation}\label{59R2}
\forall n>n_0(\ep),\;\{z\in\cD(\zeta;\ep)\mid f_n(z)=0\}\neq\emptyset.
\end{equation}
\end{enumerate}
\end{lemma}
\begin{proof}
Since $\{z\in\C\mid f(z)=0\}$ is finite, there exists
$\ep_0>0$ such that
\[\overline{\cD(\zeta,\ep_0)}\cap\{z\in\C\mid f(z)=0\}\subseteq\{\zeta\}.\]
For $\ep\in(0,\ep_0]$, define
\[c_\ep=\min\{|f(z)|\mid z\in\C,\;|z-\zeta|=\ep\}.\]
Then $c_\ep>0$, and hence
there exists $m_0(\ep)\in\N$ such that
\[\forall n>m_0(\ep),\;
\max\{|f_n(z)-f(z)|\mid z\in\C,\;|z-\zeta|=\ep\}<c_\ep.\]
This implies that, for $n>m_0(\ep)$,
$|f_n(z)-f(z)|<|f(z)|$ for $z\in\C$, $|z-\zeta|=\ep$.
By Rouch\'e's theorem \cite[Theorem~1.3.7]{RS}, we have, as multisets,
\begin{equation}\label{59R1}
\forall n>m_0(\ep),\;|\{z\in\cD(\zeta,\ep)\mid f(z)=0\}|=
|\{z\in\cD(\zeta,\ep)\mid f_n(z)=0\}|.
\end{equation}

\ref{R1}$\Rightarrow$\ref{R2}.
Let $\ep>0$. Define $n_0(\ep)=m_0(\min\{\ep,\ep_0\})$, and let $n>n_0(\ep)$.
Since $\zeta\in\{z\in\cD(\zeta,\ep)\mid f(z)=0\}$, the left-hand side
of \eqref{59R1} is nonzero. Thus $\{z\in\cD(\zeta,\ep)\mid f_n(z)=0\}
\neq\emptyset$. This proves \eqref{59R2}.

\ref{R2}$\Rightarrow$\ref{R1}.
Let $\ep\in(0,\ep_0]$. Then for $n>\max\{m_0(\ep),n_0(\ep)\}$,
\eqref{59R2} and \eqref{59R1} imply
\[\cD(\zeta;\ep)\cap\{z\in\C\mid f(z)=0\}\neq\emptyset.\]
Since $\ep\in(0,\ep_0]$ was arbitrary, we conclude $f(\zeta)=0$.
\end{proof}

\begin{lemma}\label{lem:59d}
Let $(f_n(z))_{n\in\N}$ be a sequence of real-rooted polynomials of bounded degree.
Suppose that this sequence converges to a nonzero real-rooted polynomial $f(z)$
coefficient-wise, and the limit
\[\delta=\lim_{n\to\infty}\min\{x\in\R\mid f_n(x)=0\}\]
exists. Then $f(x)$ has a real root and
\[\delta=\min\{x\in\R\mid f(x)=0\}.\]
\end{lemma}
\begin{proof}
By Lemma~\ref{lem:59b}\ref{R2}$\Rightarrow$\ref{R1}, we obtain $f(\delta)=0$.
Suppose $\zeta<\delta$ and $f(\zeta)=0$. Then by
Lemma~\ref{lem:59b}\ref{R1}$\Rightarrow$\ref{R2}, 
there exists $n_0\in\N$ such that
\[\forall n>n_0,\;\{x\in\cD(\zeta;(\delta-\zeta)/2)\mid f_n(x)=0\}\neq\emptyset.\]
Since $f_n(z)$ are real-rooted, this implies
\[\forall n>n_0,\;\exists x<\frac{\zeta+\delta}{2},\;
f_n(x)=0.\]
Thus, we obtain
\[\lim_{n\to\infty}\min\{x\in\R\mid f_n(x)=0\}\leq\frac{\zeta+\delta}{2}
<\delta,\]
which contradicts the assumption.
\end{proof}

\begin{lemma}\label{lem:59e}
Let
\[g(t,z)=h(z)t^N+\sum_{i=0}^{N-1}h_i(z)t^i\in\C[t,z]\]
be a polynomial, where $h(z)$ is a nonzero real-rooted polynomial.
Suppose that the sequence $(g(t,z))_{t\in\N}$ of polynomials in $z$
satisfies the following conditions:
\begin{enumerate}
\item\label{g1} $g(t,z)$ is real-rooted for all $t\in\N$,
\item\label{g2} $\delta=\lim_{t\to\infty}\min\{x\in\R\mid g(t,x)=0\}$ exists.
\end{enumerate}
Then
\[\delta=\min\{x\in\R\mid h(x)=0\}.\]
\end{lemma}
\begin{proof}
Let
\[f_t(z)=\frac{1}{t^N}g(t,z)=h(z)+\sum_{i=0}^{N-1}h_i(z)t^{i-N}
\quad(t\in\N).\]
Then $(f_t(z))_{t\in\N}$ is a sequence of real-rooted polynomials of bounded
degree, and it converges to $h(z)$ coefficient-wise.
The result follows from Lemma~\ref{lem:59d}.
\end{proof}

\begin{lemma}[{\cite[Lemma~3.32]{Zhan}}]\label{lem:1x}
Let
\[\tilde{A}=
\begin{bmatrix}A&B\\ B^*&D\end{bmatrix}\]
be a Hermitian matrix, and suppose $D$ is positive definite. Then
$\tilde{A}$ is positive semidefinite if and only if
$A-BD^{-1}B^*$ is positive semidefinite.
\end{lemma}

\begin{lemma}\label{lem:59c}
Let $s\geq\sqrt{2}$, and let $L\in M_{m\times n}(\C)$.
Let $D\in M_n(\C)$ be a positive definite Hermitian matrix.
If $\ell$ is the largest eigenvalue of $L^* L$, then
\[\lambda_{\min}\left(\begin{bmatrix}0&sL\\sL^*&(s^2-1)D\end{bmatrix}\right)
>-2\ell.\]
\end{lemma}
\begin{proof}
Let $\lambda_1$ denote the left-hand side.
If $L=0$, then the inequality holds trivially, so assume $L\neq0$.
Then
the matrix in question contains a $2\times 2$ matrix with negative determinant.
Thus $\lambda_1<0$.
Let $\mu^1$ be the smallest eigenvalue of $D$.
For $\lambda<0$, Lemma~\ref{lem:1x} implies
\begin{align*}
\begin{bmatrix}
-\lambda I&sL\\
sL^*& (s^2-1-\lambda)D\end{bmatrix}\succeq0
&\iff
(s^2-1-\lambda)D+\frac{s^2}{\lambda}L^* L\succeq0
\nexteqv
0\succeq s^2L^* L+\lambda(s^2-1-\lambda)D
\nexteqv
0\geq \ell s^2+\lambda(s^2-1-\lambda)\mu^1
\nexteqv
\lambda^2-(s^2-1)\lambda-\frac{\ell s^2}{\mu^1}\geq 0.
\end{align*}
Thus
\begin{align*}
\lambda_1&=\max\{\lambda\in\R\mid \lambda<0,\;
\begin{bmatrix}
-\lambda I&sL\\
sL^*& (s^2-1-\lambda)I\end{bmatrix}\succeq0\}
\nexteq
\max\{\lambda\in\R\mid\lambda<0,\;
\lambda^2-(s^2-1)\lambda-\frac{\ell s^2}{\mu^1}\geq 0\}
\nexteq
\frac{s^2-1-\sqrt{(s^2-1)^2+4\ell s^2/\mu^1}}{2}
\nexteq
-\frac{2\ell s^2}{s^2-1+\sqrt{(s^2-1)^2+4\ell s^2/\mu^1}}
\\&>
-\frac{2\ell s^2}{s^2-1+\sqrt{(s^2-1)^2}}
\nexteq
-\frac{\ell s^2}{s^2-1}
\\&\geq-2\ell,
\end{align*}
since $s^2\geq2$.
\end{proof}

\begin{theorem}\label{lem:2.2}
Let $A\in M_m(\C)$ and $D\in M_n(\C)$ be Hermitian matrices, $L\in M_{m\times n}(\C)$.
Assume $D$ is positive definite.
For $t\in\N$, 
denote by $\allone_t$
the row vector of dimension $t$ all of whose entries are $1$,
and define
\begin{equation}\label{59a}
A_t=\begin{bmatrix}
A&L\otimes\allone_t\\
L^*\otimes\allone_t^\top& D\otimes(J_t-I_t)
\end{bmatrix}.
\end{equation}
Then
\[\lim_{t\to\infty}\lambda_{\min}(A_t)=\lambda_{\min}(A-LD^{-1}L^*).\]
\end{theorem}
\begin{proof}
Clearly, $A_t$ is a principal submatrix of $A_{t+1}$. Thus
\[\lambda_{\min}(A_t)\geq\lambda_{\min}(A_{t+1}).\]

Next we show that the sequence $(\lambda_{\min}(A_t))_{t\in\N}$ is bounded from below.
Indeed, let
\begin{equation}\label{2.13}
g(t,z)=\det
\begin{bmatrix}A-z I&tL\\ L^*&(t-1)D-zI\end{bmatrix}
\in\C[t,z].
\end{equation}
Since
\[g(t,z)=\det\left(\begin{bmatrix}A&\sqrt{t}L\\
\sqrt{t}L^*& (t-1)D\end{bmatrix}-z I\right),\]
the univariate polynomial $g(t,z)$ in $z$ is real-rooted for all $t\in\N$.
For $t\geq2$, we have
\begin{align}
\min\{\lambda\in\R\mid g(t,\lambda)=0\}
&=\lambda_{\min}\left(\begin{bmatrix}
A&\sqrt{t}L\\\sqrt{t}L^*& (t-1)D\end{bmatrix}\right)
\notag\\&\geq
\lambda_{\min}(A)+\lambda_{\min}\left(\begin{bmatrix}0&\sqrt{t}L\\
\sqrt{t}L^*& (t-1)D\end{bmatrix}\right)
\notag\\&>
\lambda_{\min}(A)-2\ell
&&\text{(by Lemma~\ref{lem:59c})},
\label{59g1}
\end{align}
where $\ell$ denotes the largest eigenvalue of $L^*L$.
Let $\mu^1$ be the smallest eigenvalue of $D$. By the assumption,
we have $\mu^1>0$.
Using the block decomposition \eqref{59a}
(see \cite[Sect.~2.3]{BH}), we obtain
\begin{align*}
\lambda_{\min}(A_t)&=
\min\left\{
\Spec\begin{bmatrix}
A&tL\\
L^*& (t-1)D\end{bmatrix}
\cup\Spec
\begin{bmatrix}
A&0\\
0& D\otimes(J_t-I_t)\end{bmatrix}
\right\}
\nexteq
\min\left\{
\min\{\lambda\in\R\mid g(t,\lambda)=0\},
\lambda_{\min}(A),-\mu^1\right\}
\refby{2.13}
\\&\geq
\min\{\lambda_{\min}(A)-2\ell,-\mu^1\}
\refby{59g1}.
\end{align*}

Now, we have shown that the limit
\[\delta=\lim_{t\to\infty}\min\{\lambda\in\R\mid g(t,\lambda)=0\}\]
exists.
Since
\begin{align*}
g(t,\lambda)&=t^n
\det
\begin{bmatrix}A-\lambda I&L\\ L^*&(1-\frac{1}{t})D-\frac{\lambda}{t}I\end{bmatrix}
\nexteq
\det\begin{bmatrix}A-\lambda I&L\\ L^*&D\end{bmatrix}t^n
+\sum_{i=0}^{n-1}h_i(z)t^i
\end{align*}
for some polynomials $h_i(z)$, Lemma~\ref{lem:59e} implies
\begin{align*}
\delta&=\min\{\lambda\in\R\mid
\det\begin{bmatrix}A-\lambda I&L\\ L^*&D\end{bmatrix}=0\}
\nexteq
\min\{\lambda\in\R\mid \det(A-\lambda I-LD^{-1}L^*)=0\}
\nexteq
\lambda_{\min}(A-LD^{-1}L^*).
\end{align*}
\end{proof}

\begin{corollary}[{\cite[Lemma~2.2]{H-GEOM}}]
Let $G$ be a simple graph with $m$ vertices, and let $A$ be the
adjacency matrix of $G$.
Let $L$ be an $m\times n$ matrix with entries in $\{1,-1,0\}$.
For $t\in\N$, define $A_t$ by \eqref{59a}, where $D=I_n$.
Then
\[\lim_{t\to\infty}\lambda_{\min}(A_t)=\lambda_{\min}(A-LL^*).\]
\end{corollary}

\section{Proof of Theorem~\ref{thm:1.1S}}\label{s:Proof}

\begin{definition}\label{df:Hg}
A \textbf{\sHg} is a pair $\Ho=(H,\mu)$ where
$H=(V,E^+,E^-)$ is a signed graph with vertex set $V$ and
a labeling map $\mu:V\to\{f,s\}$,
satisfying the following conditions:
\begin{enumerate}
\item \label{en:Hg1}
every vertex with label $f$ is adjacent to at least
one vertex with label $s$;
\item \label{en:Hg2}
vertices with label $f$ are pairwise non-adjacent.
\end{enumerate}
We call a vertex with label $s$ a \textbf{slim vertex}, and
a vertex with label $f$ a \textbf{fat vertex}. We denote by
$V_s = V_s(\Ho)$ (resp.\ $V_f(\Ho)$) the set of slim (resp.\ fat)
vertices of $\Ho$.
If $E^-=\emptyset$, then we call $\Ho$ an \textbf{unsigned Hoffman graph},
or simply, a \textbf{Hoffman graph}.
\end{definition}

For a \sHg\ $\Ho$, let $A$ be its adjacency matrix,
\begin{equation}\label{A}
A=
\begin{bmatrix}
A_s & C \\
C^T & O
\end{bmatrix}
\end{equation}
in a labeling in which the fat vertices come last.
The signed graph with adjacency matrix $A_s$ is called the
\textbf{slim subgraph} of $\Ho$.
\textbf{Eigenvalues} of $\Ho$ are the eigenvalues of the
real symmetric matrix $B(\Ho)=A_s-CC^T$.
It is easy to see that, if two \sHg\ are switching equivalent
as signed graphs, then they have the same set of eigenvalues as
\sHg s.
Let $\lambda_{\min}(\Ho)$ denote the smallest eigenvalue of $\Ho$.

For a \sHg\ $\Ho$ and a positive integer $t$,
we denote by $G(\Ho,t)$ the signed graph
obtained by replacing
every fat vertex of $\Ho$
by $K_t$ consisting of $(+)$-edges.

\begin{theorem} \label{thm:limit}
Let $\Ho=(H,\mu)$ be a \sHg.
Then
\[\lim_{t\to\infty}\lambda_{\min}(G(\Ho,t))=\lambda_{\min}(\Ho).\]
\end{theorem}
\begin{proof}
If the adjacency matrix of $\Ho$ is given by \eqref{A}, then the
adjacency matrix of $G(\Ho,t)$ is \eqref{59a}, where $D=I$.
Thus, the result is immediate from Theorem~\ref{lem:2.2}.
\end{proof}

\begin{figure}[h]
\begin{center}
\begin{tikzpicture}
\filldraw [black]
(1,0) circle (12pt)
(1,2) circle (4pt) ;
\draw[thick] (1,0) -- (1,2);
\node (K) at (1,-1) {$\Ho_1$};
\end{tikzpicture}
\qquad
\begin{tikzpicture}
\filldraw [black]
(0,0) circle (12pt)
(2,0) circle (12pt)
(1,2) circle (4pt) ;
\draw[thick] (0,0) -- (1,2);
\draw[thick] (2,0) -- (1,2);
\node (K) at (1,-1) {$\Ho_2$};
\end{tikzpicture}
\qquad
\begin{tikzpicture}
\filldraw [black]
(1,0) circle (12pt)
(2,2) circle (4pt)
(0,2) circle (4pt) ;
\draw[thick] (1,0) -- (2,2);
\draw[thick] (1,0) -- (0,2);
\node (K) at (1,-1) {$\Ho_3$};
\end{tikzpicture}
\qquad
\begin{tikzpicture}
\filldraw [black]
(0,0) circle (12pt)
(0,2) circle (3pt)
(2,2) circle (3pt)
(2,0) circle (12pt) ;
\draw[thick] (0,0) -- (0,2);
\draw[thick] (2,0) -- (2,2);
\draw[thick] (2,2) -- (0,2);
\node (K) at (1,-1) {$\Ho_4$};
\end{tikzpicture}
\caption{The Hoffman graphs $\Ho_i$ ($i=1,2,3,4$)}\label{fig:234}
\end{center}
\end{figure}
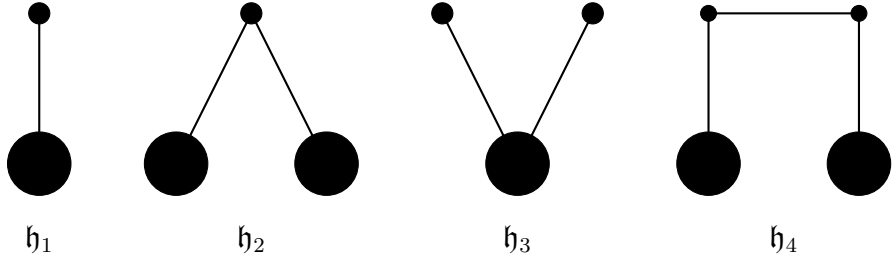

Following \cite{hlg},
let $\Ho_1,\Ho_2,\Ho_3,\Ho_4$ denote the (unsigned) Hoffman graph
defined in Fig.~\ref{fig:234}.
Then  $\lambda_{\min}(\Ho_1)=-1$, and
 $\lambda_{\min}(\Ho_i)=-2$ for $i=2,3,4$.

For the remainder of this section, we
fix a real number $\lambda$ with $-2<\lambda<-1$.
By Theorem~\ref{thm:limit}, there exist
$n_0\in\N$
such that
\begin{equation}\label{1a123}
\lambda>\lambda_{\min}(G(\Ho_i,n_0))\quad(i=2,3,4).
\end{equation}

\begin{lemma}\label{lem:5a}
Let $S$ be a signed graph.
If $\lambda_{\min}(S)\geq\lambda$ and $S$ is switching equivalent to
an unsigned graph $G$, then
\begin{align}
G&\not\supset G(\Ho_i,n_0)\quad(i=2,3,4).
\label{7a1}
\end{align}
\end{lemma}
\begin{proof}
Since $\lambda_{\min}(S)=\lambda_{\min}(G)$, the assertion is
immediate from \eqref{1a123}.
\end{proof}

\begin{lemma}\label{lem:i}
Let $S$ be a signed graph.
If $\lambda\leq\lambda_{\min}(S)<-1$ and $S$ is switching equivalent to
the line graph $\LG(H)$ of some connected graph
$H$, then
$\delta(S)<2n_0$.
\end{lemma}
\begin{proof}
Let $G=\LG(H)$.
Since $\lambda_{\min}(S)<-1$, $S$ is not switching equivalent to a complete graph.
This implies that $G$ is not complete, and hence
$H$ is not a claw.
Since $H$ is not a claw, there exists
an edge $x=\{u,v\}\in E(H)$ with $\deg_H(u),\deg_H(v)>1$.

Suppose, to the contrary, that $\delta(S)\geq 2n_0$. Then
$\delta(G)\geq 2n_0$, and therefore
$\deg_H(u)+\deg_H(v)\geq 2n_0+2$.
If $\min\{\deg_H(u),\deg_H(v)\}\geq n_0+1$, then $N_G(x)$ consists of two
connected components each of which contains
$K_{n_0}$.
This contradicts \eqref{7a1}. Thus, we may assume without loss of generality
that $\deg_H(v)<n_0+1$. Then $\deg_H(u)>n_0+1$. Since $\deg_H(v)>1$, there exists
a vertex $w\in N_H(v)\setminus\{u\}$. Let $y$ denote the edge $\{v,w\}$.
Since $\deg_{G}(y)\geq 2n_0$, we have
$\deg_H(w)>n_0+1$. This implies that
the subgraph of $G$ induced by $\{x,y\}\cup N_G(x)\cup N_G(y)$
contains $G(\Ho_2,n_0)$,
contradicting \eqref{7a1}.
\end{proof}

\begin{proof}[Proof of Theorem~\ref{thm:1.1S}]
Since there are only finitely many exceptional graphs, there exists
a positive integer $d_0$ such that every exceptional graph has
minimum degree bounded by $d_0$.

Recall that we have fixed $\lambda\in(-2,-1)$. We define the value
$f(\lambda)$ of the function $f:(-2,-1)\to\R$ by
\[f(\lambda)=\max\{2n_0,d_0\}+1.\]
Let $S$ be a connected signed graph with $\lambda_{\min}(S)\geq\lambda$ and
$\delta(S)\geq f(\lambda)$. Since $\delta(S)>d_0$, we see that $S$ is not
exceptional. This means that $S$ is integrally represented.
By Corollary~\ref{cor:1.4},
there exists a tree $H$ such that $\LG(H)$ is switching equivalent to
$S$ or $S$ with one vertex removed. In the former case, Lemma~\ref{lem:i}
implies that $\lambda_{\min}(S)=-1$, and hence $S$ is switching equivalent
to a complete graph. Suppose $\LG(H)$ is switching equivalent to
$S-u$ for some vertex $u$ of $S$. Since $\lambda_{\min}(S-u)\geq\lambda_{\min}(S)
\geq\lambda$ and $\delta(S-u)\geq\delta(S)-1\geq2n_0$,
Lemma~\ref{lem:i} implies that $\lambda_{\min}(S-u)=-1$,
and hence $S-u$ is switching equivalent to a complete graph.
Since $\delta(S)>d_0$, the vertex $u$ has degree greater than $n_1$.
If $S$ is not switching equivalent to a complete graph, then
there exists a non-neighbor $u'$ of $u$ in $S$. Then the subgraph
induced on the common neighbors of $u,u'$ together with $u,u'$ themselves
has smallest eigenvalue less than $\lambda$ by \eqref{1a123}. This
implies $\lambda_{\min}(S)<\lambda$, contrary to the assumption.
Therefore, $S$ is switching equivalent to a complete graph.
\end{proof}

\begin{lemma}\label{lem:oddodd}
Let $S$ be an odd cycle with 
an odd number of $(-)$-edges. Then $S$
has smallest eigenvalue $-2$.
\end{lemma}
\begin{proof}
The signed graph $S$ is switching equivalent to an odd cycle in
which all edges are $(-)$-edges. Since this is the negative of 
a $2$-regular graph, it has smallest eigenvalue $-2$.
\end{proof}

\begin{lemma}\label{lem:comp}
Let $S$ be a signed graph whose underlying graph $U(S)$ is complete.
If $\lambda_{\min}(S)\geq-\sqrt{2}$, then 
$S$ is switching equivalent to a complete graph.
\end{lemma}
\begin{proof}
After switching, we may assume that there
exists a vertex $x$ of $S$ such that all edges incident with $x$ are
positive. Since a triangle with one negative edge has smallest
eigenvalue $-2$ while $\lambda_{\min}(S)\geq-\sqrt{2}$, it follows
that $S$ cannot contain such a triangle. This implies that 
all edges of $S$ not containing $x$ are positive. Therefore,
$S$ itself is a complete graph.
\end{proof}

\begin{proposition}\label{prop:Seidel}
Let $S$ be a connected signed graph with smallest eigenvalue greater than
$-\sqrt{2}$. Then $S$ is switching equivalent to a complete graph.
\end{proposition}
\begin{proof}
If the underlying graph of $S$ is not complete, then $S$ contains a signed
$2$-path, which has smallest eigenvalue $-\sqrt{2}$. This contradiction
shows that the underlying graph of $S$ must be complete.
The result then follows from Lemma~\ref{lem:comp}.
\end{proof}

As a consequence of Proposition~\ref{prop:Seidel},
the values of the function $f$ in Theorem~\ref{thm:1.1S} on the interval
$(-\sqrt{2},-1)$ can be arbitrary, since the conclusion of
Theorem~\ref{thm:1.1S} holds for $\lambda\in(-\sqrt{2},-1)$ without
any assumption on the minimum degree $\delta(S)$.

A natural question is to determine the smallest possible value
of $f(-\sqrt{2})$ so that $\delta(S)\geq f(-\sqrt{2})$ and
$\lambda_{\min}(S)\geq-\sqrt{2}$ implies that $S$ is switching
equivalent to a complete graph.
In addition to the $2$-path, there is another signed graph
with smallest eigenvalue $-\sqrt{2}$, namely, a $4$-cycle
with one $(-)$-edge. It has adjacency matrix
\[ A=\begin{bmatrix} 0&1&0&-1\\ 1&0&1&0\\ 0&1&0&1\\ -1&0&1&0\end{bmatrix},\]
and we have $A^2=2I$. 
Its underlying graph is regular of valency $2$. So we must take
$f(-\sqrt{2}) > 2$ to exclude this graph.
In fact, $f(-\sqrt{2}) = 3$ does give the correct conclusion.

\begin{proposition}\label{prop:12}
If $S$ is a connected signed graph with $\lambda_{\min}(S)\geq-\sqrt{2}$
and $\delta(S)\geq 3$, then $S$ is
switching equivalent to a complete graph (and hence $\lambda_{\min}(S) = - 1$).
\end{proposition}
\begin{proof}
We note first that it suffices to show that the underlying graph $U(S)$
of $S$ is complete,
by Lemma~\ref{lem:comp}.
Since $S$ has minimum degree at least $3$,
$S$ must contain a cycle (since otherwise 
$S$ is a tree,
having a leaf, meaning minimum degree is 1). 

Suppose first $S$ contains a
triangle. If $U(S)$ is a triangle, then $U(S)$ is a complete graph,
so we are done by the first paragraph.
Otherwise,
$S$ has at least $4$ vertices, so
$U(S)$ must contain a triangle with one pendant edge attached, or 
$K_{1,1,2}$.
By Lemma~\ref{lem:oddodd}, $S$ can be switched to contain
one of the two graphs with all positive edges.
The former has smallest eigenvalue $\approx-1.48$, the 
latter
$\approx-1.56$, both are strictly less than $-\sqrt{2}$. 
This is a contradiction. Thus $S$ has no triangle,
so $S$ contains a
cycle of length at least $4$. $S$ cannot contain
a 
cycle of length at least $4$ with all edges positive. So $S$
contains  a
cycle of length at least $4$ with odd number of negative
edges. 
If the length is odd, then we get a contradiction by Lemma~\ref{lem:oddodd}.
So the length is even. If the
length is at least $6$, then it contains 
a
path with $5$ vertices.
A path with $5$ vertices has smallest eigenvalue $-\sqrt{3} < -\sqrt{2}$,
a contradiction. So the only possible cycle is a $4$-cycle with one
negative edge. 
Since $\delta(S)\geq3$,
$S$ strictly contains 
a
$4$-cycle with one
negative edge. 
As $S$ cannot contain
a triangle with one pendant edge attached, or 
$K_{1,1,2}$, we see that 
$S$ contains $K_{1,3}$ with smallest eigenvalue
$-\sqrt{3} < -\sqrt{2}$, a contradiction.
\end{proof}

In the original setting of Hoffman's Theorem~\ref{thm:H}, 
a much easier argument than the above shows that 
we may define $h(-\sqrt{2})=2$.

\section{A signed analogue of generalized line graphs}\label{s:genline}

\begin{definition}\label{dfn:line_sigraph}
Given a signed graph $S=(V,E^+,E^-)$, 
the \textbf{line signed graph} $\LG(S)$ is the signed graph
with vertex set $E^+\cup E^-$, and two distinct vertices are joined by
a signed edge if they are incident and the sign is the product of their signs.
\end{definition}

The definition of line signed graph can be best understood in terms of
signed incidence matrix. The \textbf{signed incidence matrix}
$B$ of a signed graph $S=(V,E^+,E^-)$ is the matrix
whose rows and columns are indexed by $V$ and $E^+\cup E^-$ respectively, 
such that its $(i,e)$-entry is equal to the sign of $e$ if $i\in e$, and otherwise
$0$. The adjacency matrix of $\LG(S)$ is then given by
$B^\top B-2I$. Note that for an unsigned graph $S$, $\LG(S)$ is nothing
but the ordinary signed graph of the graph $S$.

Line signed graphs (or sometimes called signed line graphs)
have been considered in \cite{BS,GHZ},
but our definition is different from those introduced there.
We list some properties of $\LG(S)$ together with comments
pertaining to the corresponding properties of line graphs defined
in \cite{BS,GHZ}. Note that, we denote by $-S$ the \textbf{negative} of
a signed graph $S$, which is obtained by exchanging $E^+$ and $E^-$ in $S$.
\begin{enumerate}
\item\label{L1} If $S$ is an unsigned graph, that is, $E^-=\emptyset$, then
$\LG(S)$ coincides with the ordinary line graph.
This is not true in \cite{BS,GHZ}.
Indeed, $\LG(S)$ is the line graph of $-S$ in the sense of \cite{BS},
while it is the negative of the line graph of $-S$ in the sense of \cite{GHZ}.
\item\label{L2} The line signed graph
$\LG(S)$ has smallest eigenvalue at least $-2$.
This is true in \cite{BS} but not in \cite{GHZ}.
\item\label{L3} The line signed graph
$\LG(S)$ is uniquely determined by $S$.
This is not true in \cite{BS,GHZ}, where the signed line graphs
depends on the choice of an orientation, and defined 
only up to switching equivalence.
\item\label{L4} If $S$ and $S'$ are switching equivalent, so are $\LG(S)$
and $\LG(S')$.
This is true in \cite{BS,GHZ}.
\end{enumerate}

To see the property \ref{L2}, recall that
the adjacency matrix of $\LG(S)$ is given by
$B^\top B-2I$, where $B$ is the signed incidence matrix of $S$.
Since $B^\top B$ is positive semidefinite, \ref{L2} holds.

The property \ref{L4} follows from a stronger claim that
$\LG(S)$ is switching equivalent to the line graph of the
underlying graph $U(S)$ of $S$. 
To see this, observe that every cycle in $\LG(S)$ contains
an even number of $(-)$-edges, and then invoke
\cite[Prop.~3.2]{Zas}. This stronger statement
indicates that 
spectral consideration on line signed graphs in our sense
reduces to that of line (unsigned) graphs. 
However, such a reduction will not occur in 
the generalization to follow 
(see Definition~\ref{dfn:h-line_sigraph} and 
comments after that).

Given a signed graph $S=(V(S),E^+(S),E^-(S))$, 
for convenience, we denote by $\sigma(e)$ the sign of an edge $e
\in E^+(S)\cup E^-(S)$. 
We can construct a Hoffman signed graph $\Ho=(H,\mu)$, where
$H=(V(H),E^+(H),E^-(H))$ is a signed graph, as follows. Define
\begin{align*}
V(H)&=V(S)\cup E^+(S)\cup E^-(S),\\
E^\epsilon(H)&=\{\{i,e\}\mid i\in V(S),\;e\in E^\epsilon(S),\;i\in e\}
\\&\quad\cup\{\{e,e'\}\mid e,e'\in E^+(S)\cup E^-(S),\;
|e\cap e'|=1,\;
\sigma(e)\sigma(e')=\epsilon\},
\end{align*}
where $\epsilon=\pm1$,
and $\mu:V(H)\to\{f,s\}$ by
\[\mu(x)=\begin{cases}
f&\text{if $x\in V(S)$,}\\
s&\text{otherwise.}
\end{cases}\]
Then the slim subgraph of $H$ coincides with the line signed graph of $S$
defined in Definition~\ref{dfn:line_sigraph}.
Note that every slim vertex of $H$ has exactly two fat neighbors, joined
by edges of the same sign. In other words, $H$ is obtained by gluing
Hoffman graphs $\Ho_2$ (see Figure~\ref{fig:234}) and $\Ho_2^{--}$
(see Figure~\ref{fig:2-}) appropriately. The adjacency of two slim
vertices occurs exactly when they have a fat neighbor in common, and
the sign of the edge connecting them is also determined by the 
sign of edges connecting them to fat neighbors.
This observation motivates the definitions to follow.

Let $\Ho=(H,\mu)$ be a \sHg, where $H=(V,E^+,E^-)$.
For a slim vertex $x$ of $\Ho$,
the \textbf{representing vector} of $x$ is the vector $\varphi(x)$ indexed by
the set of fat vertices, defined by
\[\varphi(x)_z=\begin{cases}
1&\text{if $\{x,z\}\in E^+$,}\\
-1&\text{if $\{x,z\}\in E^-$,}\\
0&\text{otherwise.}
\end{cases}\]
If the \sHg\ $\Ho=(H,\mu)$ is obtained from a signed graph $S$ as above, then
for two distinct slim vertices $x$ and $y$ of $H$, 
the sign of the edge $\{x,y\}$ is the inner product $(\varphi(x),\varphi(y))$.
We will axiomatize this to define decompositions of Hoffman signed graph.
Note that this concept has already been considered by Woo and
Neumaier \cite{hlg} for unsigned Hoffman graphs.

For a vertex $x$ of $\Ho$ we define $N^f(x) = N^f_{\Ho}(x)$ (resp.\
$N^s(x) = N^s_{\Ho}(x)$) the
set of fat (resp.\ slim) neighbors of $x$ in $\Ho$.
The set of all
neighbors of $x$ is denoted by $N(x) = N_{\Ho}(x)$,
	that is $N(x) = N^f(x) \cup N^s(x)$.
In a similar fashion,
for vertices $x$ and $y$
	we define $N^f(x,y)=N^f_{\Ho}(x,y)$ to be the
set
of common fat neighbors of $x$ and $y$.

A \textbf{decomposition} of a \sHg\ $\Ho$
is a family $\{\Ho^i\}_{i=1}^n$ of non-empty
induced Hoffman subgraphs of $\Ho$
satisfying the following conditions:
\begin{enumerate}
\item\label{d1}
$V(\Ho)=\bigcup_{i=1}^n V(\Ho^i)$;
\item\label{d2}
$V^s(\Ho^i) \cap V^s(\Ho^j)=\emptyset$ if $i\neq j$;
\item\label{d3}
For each $x\in V^s(\Ho^i)$,
$N^f_{\Ho}(x) \subseteq V^f(\Ho^i)$
\item\label{d4}
If $x \in V^s(\Ho^i)$, $y\in V^s(\Ho^j)$, and $i \neq j$,
the inner product $(\varphi(x),\varphi(y))$ is $1,-1,0$,
according as $\{x,y\}$ is a $(+)$-edge, $(-)$-edge, or non-edge.
\end{enumerate}
A \sHg\ $\Ho$ is said to be \textbf{decomposable} if
$\Ho$ has a decomposition $\{\Ho^i\}_{i=1}^n$ with $n \geq 2$,
and $\Ho$ is said to be \textbf{indecomposable}
if $\Ho$ is not decomposable.

\begin{definition}\label{dfn:h-switch}
Two \sHg s $\Ho$ and $\Ho'$ are \textbf{switching equivalent} if
$\Ho'$ can be obtained from $\Ho$ by
switching with respect to a subset of slim vertices of $\Ho$.
\end{definition}

For example, $\Ho_2$ is switching equivalent to $\Ho_2^{--}$, but
not to $\Ho_2^-$.

\begin{definition}\label{dfn:h-line_sigraph}
Let $\h$ be a family of
switching classes
of \sHg s.
An \textbf{$\h$-line signed graph} is
an induced Hoffman subgraph of a \sHg\
which has a decomposition $\{\Ho^i\}_{i=1}^n$
such that the
switching
class of $\Ho^i$ belongs to
$\h$ for all $i=1,\dots,n$.
\end{definition}

It is clear from the definition that the line signed graphs
are precisely the slim subgraphs of a Hoffman signed graph which admits a decomposition
all of whose components are isomorphic to $\Ho_2$ or $\Ho_2^{--}$.
Since $\Ho_2^{--}$ is switching equivalent to
$\Ho_2$, this means that every line signed graph
is an $\h$-line signed graph, where
$\h$ is the singleton set consisting of the switching class of $\Ho_2$. 
This is precisely the stronger claim mentioned in the proof
of the property \ref{L4} of line signed graphs.

We note that, however, some $\h$-line signed graph are not
switching equivalent to an unsigned graph for some other family
$\h$.  For,
let $\Ho$ be the Hoffman graph with one fat vertex having two slim
neighbors connected by a $(-)$-edge, where the edges connecting
the fat vertex and slim vertices are $(+)$-edges. Let
$\h=\{[\Ho_1],[\Ho]\}$ (see Fig.~\ref{fig:234} for the definition
of $\Ho_1$).
Then the slim graph of an 
$\h$-line signed graph
obtained by identifying the fat vertices of $\Ho_1$ and 
$\Ho$ is the
triangle with only one $(-)$-edge. Thus, it is not switching 
equivalent to an unsigned graph.

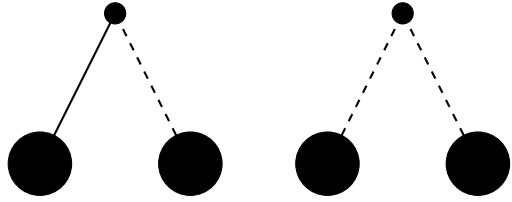
\begin{figure}
\begin{center}
\begin{tikzpicture}
\filldraw [black]
(0,0) circle (12pt)
(2,0) circle (12pt)
(1,2) circle (4pt) ;
\draw[thick] (0,0) -- (1,2);
\draw[thick,dashed] (2,0) -- (1,2);
\end{tikzpicture}
\qquad
\begin{tikzpicture}
\filldraw [black]
(0,0) circle (12pt)
(2,0) circle (12pt)
(1,2) circle (4pt) ;
\draw[thick,dashed] (0,0) -- (1,2);
\draw[thick,dashed] (2,0) -- (1,2);
\end{tikzpicture}
\caption{The Hoffman signed graphs $\Ho_2^-,\Ho_2^{--}$}\label{fig:2-}
\end{center}
\end{figure}

\begin{proposition}\label{prop:H3H2-}
Let $\h$ consist of 
three switching classes $[\Ho_2],[\Ho_2^{-}],[\Ho_3]$.
Then every $\h$-line signed graph is
switching equivalent to
an $\{[\Ho_2],[\Ho_2^{-}]\}$-line signed graph.
\end{proposition}
\begin{proof}
Suppose that a signed Hoffman graph $\Ho$ is an $\h$-line signed graph.
If a decomposition of $\Ho$
contains a summand $\Ho^i$ which is switching equivalent to
$\Ho_3$, then we may apply switching with respect to one or both of
slim vertices of $\Ho_i$ to make $\Ho^i$ isomorphic to $\Ho_3$.
Having done this, we can add a new common fat neighbor
to the two slim vertices of $\Ho^i$, where one of the edge connecting a slim vertex
to the new fat neighbor is a $(-)$-edge, the other is a $(+)$-edge.
Doing this process for each summand isomorphic to $\Ho_3$, we can convert
$\Ho_3$ to the sum of $\Ho_2$ and $\Ho_2^-$.
\end{proof}

The conclusion of Proposition~\ref{prop:H3H2-} cannot be changed to claim
the switching equivalence to an $\{[\Ho_2],[\Ho_3]\}$-line signed graph,
or to an $\{[\Ho_2^-],[\Ho_3]\}$-line signed graph. Indeed,
a triangle consisting of three $(-)$-edges is  an $\{[\Ho_2^-]\}$-line
singed graph which is not switching equivalent to an $\{[\Ho_2],[\Ho_3]\}$-line signed graph,
and a pentagon consisting of five $(+)$-edges is  an $\{[\Ho_2]\}$-line
singed graph which is not switching equivalent to an $\{[\Ho_2^-],[\Ho_3]\}$-line signed graph.

The meaning of Proposition~\ref{prop:H3H2-} for unsigned slim graphs
is as follows. If $G$ is an unsigned slim graph which is 
an $\{[\Ho_2],[\Ho_2^{-}],[\Ho_3]\}$-line signed graph, then it is
an $\{[\Ho_2],[\Ho_3]\}$-line graph, and hence a generalized
line graph by \cite[Example~2.1]{hlg}. Generalized line graphs
are also known \cite{CGSS} as graphs represented by the root system
\[D_\infty=\{\pm \mathbf{e}_i\pm \mathbf{e}_j\mid 1\leq i<j<\infty\}.\]
This fact becomes transparent by Proposition~\ref{prop:H3H2-}
since clearly, $\{[\Ho_2],[\Ho_2^-]\}$-line signed graphs 
are precisely the graphs represented by $D_\infty$,
in terms of representing vectors. Such a representation shows that
every slim $\{[\Ho_2],[\Ho_2^-]\}$-line signed graph has smallest
eigenvalue at least $-2$. Alternatively, this is a consequence of
the following proposition which generalizes 
the more general fact about decomposition of Hoffman signed graphs.

\begin{proposition}\label{prop:sum}
Suppose a \sHg\ $\Ho$ has a decomposition $\{\Ho^i\}_{i=1}^n$.
Let
\[A=\begin{bmatrix}A_s&C\\ C^T&O\end{bmatrix}\text{ and }
A_i=\begin{bmatrix}A_s^{(i)}&C_i\\ C_i^T&O\end{bmatrix}\]
be the adjacency matrices of $\Ho$ and
$\Ho^i$ in a labeling in which the fat vertices
come last.
Then $A_s-CC^T$ is the direct sum of matrices
$A_s^{(i)}-C_iC_i^T$ for $i=1,\dots,n$.
In particular,
\[\lambda_{\min}(\Ho)=\min\{\lambda_{\min}(\Ho^i) \mid 1 \leq i \leq n\}. \]
\end{proposition}
\begin{proof}
Let $x \in V^s(\Ho^i)$, $y\in V^s(\Ho^j)$.
Suppose $i \neq j$. Then by the condition \ref{d4} of decomposition,
the $(x,y)$-entry of $A_s$ coincides with the inner product
$(\varphi(x),\varphi(y))$ which is the $(x,y)$-entry of
$CC^T$. Thus the $(x,y)$-entry of $A_s-CC^T$ is $0$, which is the
same as the corresponding entry of the diagonal join.

Suppose $i=j$. Since $\Ho^i$ is an induced subgraph of $\Ho$,
the submatrix of $A_s$ corresponding to the $V^s(\Ho^i)$ is
exactly $A_s^{(i)}$. By the condition \ref{d3} of decomposition,
$\varphi(x)$ has support contained in $V^f(\Ho^i)$.
Thus, the inner product $(\varphi(x),\varphi(y))$ coincides
with the $(x,y)$-entry of $C_iC_i^T$. Therefore,
the $(x,y)$-entry of $A_s-CC^T$ is the same as the
corresponding entry of $A_s^{(i)}-C_iC_i^T$.
\end{proof}

\subsection*{Acknowledgements}
The authors thank Zoran Stani\'c for valuable comments.

\end{document}